\documentclass[11pt,reqno,twoside,a4paper]{article}
\usepackage{mysty}
\usepackage{wrapfig}
\usepackage{fullpage}

\title{Differentiation of inertial methods for optimizing smooth parametric function}

\author{Jean-Jacques Godeme\thanks{ INRIA Sophia  \& Universit\'e C\^ote d’Azur, CNRS, LJAD, France. \underline{e-mail:}\texttt{jeanjacquesgodeme@gmail.fr}} } 
  
\date{}

\begin{document}

\maketitle
\vspace*{-1.02cm}
\begin{center}
	\textit{In honor, of the 90th Birthday of Professor emiritus R. Tyrrell Rockafellar for all his contributions to
		mathematics and his role as an inspiration to younger generations.}
\end{center}

\begin{abstract}
In this paper, we consider the minimization of a $C^2-$smooth  and strongly  convex objective depending on a given parameter, which is usually found in many practical applications. We suppose that we desire to solve the problem with some  inertial methods which cover a broader existing well-known inertial methods. Our main goal is to analyze  the derivative of this algorithm as an infinite iterative process in  the sense of ``automatic'' differentiation. This procedure is  very common and has recently gained more attention. From a pure optimization perspective and under some mild premises, we show that  any sequence generated by theses inertial methods converge to the unique minimizer of the problem, which depends on the parameter. Moreover,  we show a local linear convergence rate of the generated sequence. Concerning the differentiation of the scheme, we prove that  the derivative of the sequence with respect to the parameter converges to the derivative of the limit of the sequence showing that it is <<derivative stable>>. Finally, we investigate  the rate at  which the convergence occurs. We show that, it is locally linear with an error term tending to zero. 
\end{abstract}

\section{Introduction}\label{sec-intro}

\subsection{Problem Statement \& Prior work}
In this paper, we study derivatives of an iterative process, in the ``automatic'' differentiation sense, to solve a parametric optimization problem. This approach is  one the most popular method to differentiate an  iterative algorithm generated by a computer.  Differentiating algorithms has been properly introduced in the 90's see  \cite{gilbert92,beck1994} for instance and gained recently a lot of attention within applications such as machine learning precisely hyperparameter optimization, metalearning \etc. We refer the interested reader to  \cite{griewank1993,griewank2007} for a more detailed introduction and to \cite{baydin2018} for a recent review to  \cite{iutzeler24} for the stochastic case, to \cite{bolte22,bertrand2020} for extensions to the nonsmooth setting  and the following  thesis \cite{vaiter2021}, for applications in seismic tomography \cite{liu_multimodal_2024} and reference therein.

 It is well known that for a huge amount of smooth program, the derivatives can be obtained using  the rule of derivation of  composed functions. Therefore, the main goal is to know whether or not if we have convergence of the derivatives of the iterates to the derivative of the solution of the optimization problem. The standard approach in the smooth case is to use the well-known implicit function theorem which is our concern in this paper. To be more precise, we consider the following parametric-optimization problem 
\begin{equation}\label{eq:param-optim-intro}\tag{$\calP_{\theta}$}
\min_{x\in\bbR^n}f(x,\theta)
\end{equation}
where $\theta\in\Theta \subseteq\bbR^m$ is a parameter in the problem and the objective function is smooth enough. One of the most common way to solve this problem is to use the gradient-descent  like methods or accelerated version which are known to be inertial methods (\cite{polyak_methods_1964,nesterov_method_1983,liang_activity_2017,marumo_parameter_2024}). The natural questions which come and are the core of our analysis are the following: 

\textit{Do we have convergence of the derivatives of the inertial methods to the derivatives of the solution? And if the answer is positive, can we quantify this convergence rates ?}

Despite its conceptual simplicity, the answers to these questions are not as trivial as it may seem and requires a  careful analysis before apply any differentiation toolboxes. We recall that our motivations for studying theses questions come directly from the large range of  applications  as mentioned above. The second is that, they are no proper answer to this questions to the best of our knowledge. We hope that our tentative of answer clarifies the questions and propose a good theoretical mathematical framework and background to the daily users of  differentiation methods. 

\subsection{Contributions}
In this work, we start by analyzing  the parametric optimization problem \eqref{eq:param-optim-intro} from a pure optimization perspective. We provide  conditions, on which the optimization problem has at least one solution and under strong convexity with respect to the first variable, \eqref{eq:param-optim-intro} has actually  a unique minimizer for all fixed $\theta\in\Theta$. We then consider a set of inertial methods, for solving smooth enough optimization problem. This inertial methods can be written in very concise form in a higher dimension by defining an appropriate mapping. We study the properties of this mapping, we show that this mapping is continuously differentiable, moreover it is also Lipschitz continuous. We turn to provide a  global convergence analysis to the unique minimizer of the problem under some mild premise on the inertial parameters and the stepsizes. To properly understand the behaviour of this method, we also analyze  the local convergence rate of this schemes and it turns out that this convergence rates is linear. Since, our main purpose of the paper is to differentiate this inertial method, we provide a clear and concise analysis of the derivative of the inertial methods with respect to the parameters. Firstly, we use the rule of derivation of composed functions to get an explicit formula of the derivative of the inertial methods. Then, we show that  since, the sequence globally converges to the minimizer thus under some premises that using a  standard technique  the sequence of the derivative also converges to the derivative of the limit point which is the minimizer of the problem. Finally, we provide a  convergence rates analysis of this phenomenon. Indeed, we show that we have a local linear convergence with an error term which tends to zero  as  the number of iteration $k$ tends to infty. The most important part is that we do not suppose the Lipschitz continuity of the second order derivative like most work in the literature of differentiation of algorithm for minimizing  $C^2-$smooth function. We think that this hypothesis was too restrictive and do not cover most of practical applications. 

\subsection{Paper organization}

We organized the rest of this paper  as follows. In Section~\ref{sec:opt-prblm}, we consider the parametric optimization as a pure optimization problem  and we provide a full analysis of the convergence, global and local behaviour of any sequence generated by our inertial method to the unique minimizer of the problem under some mild premise.  Section~\ref{sec:ad-ima}  is devoted to  the proper analysis of the  differentiation of our inertial methods for solving  the parametric-optimization problem. We  illustrate most of our results with numerical experiments in  Section~\ref{sec_numexp}. Finally, we postpone in Appendix some technicals proofs, remarks and a toolbox on  differentiation that we use throughout our paper.



\section{ Optimization problem of interest}\label{sec:opt-prblm}
\subsection{Existence result}\label{subsec:exist}
We recall  that we want to solve the following optimization problem \eqref{eq:param-optim}
\begin{equation*}\tag{$\calP_{\theta}$} \label{eq:param-optim}
\min_{x\in\bbR^n}f(x,\theta)
\end{equation*}
where $\theta\in\Theta \subseteq\bbR^m.$  We have the following premises on the function that we want to optimize. 

\begin{premise}{\ }\label{prem_A}
\begin{enumerate}[label=(A.\arabic*)]
 \item \label{prem_A1} $f$ is a $C^2-$smooth  function with respect to both $x,\theta$,
 \item\label{prem_A2}  the gradient $\nabla_x f$ is $L-$Lipschitz continuous in both $x,\theta$, 
 \item \label{prem_A3}  $\forall\theta\in\Theta,$  the function  $x \mapsto f(x,\theta)$ is  level bounded \ie  all its sublevel sets are bounded. 
\end{enumerate}
\end{premise}
\begin{remark}
Premise~\ref{prem_A1} is a bear  requirement  since our main target is to differentiate a first-order method in the whole space. While Premise~\ref{prem_A2} implies that the operators $\nabla_{xx}^2f$ and $\nabla_{x\theta}^2f$  have upper bounds at each points. Finally, Premise~\ref{prem_A3} ensures that the minimization problem \eqref{eq:param-optim} has at least one minimizer for each $\theta\in\Theta$. It corresponds to having the property that $\forall \theta\in\Theta, \quad f(x,\theta)\to\infty$ as $\normm{x}\to\infty.$ 
\end{remark}

As an unconstraint optimization problem, the next question that arises is whether or not this problem has a solution. It turns out that under our given Premise~\ref{prem_A}, the answer is true. 
\begin{lemma}[Existence]\label{lem:exist} Let us consider  the problem \eqref{eq:param-optim} forall $\theta\in\Theta.$ Under the Premise~\ref{prem_A}, we have
\begin{enumerate}
\item\label{lem:exist_1} $\forall\theta\in\Theta,$ the minimization problem \eqref{eq:param-optim}  has a solution \ie $\argmin_x f(x,\theta)\neq\emptyset$. 
\item\label{lem:exist_2} Moreover,  if $\forall\theta\in\Theta,$ the function $x\mapsto f(x,\theta)$ is strongly convex then $\argmin_x f(x,\theta)$ is a singleton. 
\end{enumerate}
\end{lemma} 
\begin{proof}$~$
 Forall $\theta\in\Theta,$ \eqref{eq:param-optim} is an optimization problem and thanks to Premise~\ref{prem_A1}, $f(x,\theta)$ is continuous with respect to $x$. Moreover, under Premise~\ref{prem_A2},  for any $\theta\in\Theta$ the function $x\mapsto f(x,\theta)$ is level bounded. We conclude by applying \cite[Theorem~1.6]{rockafellar_variational_1998} which prove Claim~\ref{lem:exist_1}.
 The second Claim is a standard result in convex optimization and we refer to \cite{rockafellar_convex_1970} for an extensive study of the problem.
\end{proof}

\subsection{Solving with inertial methods}\label{subsec:solv-im}

Throughout this section, we suppose that $f$ satisfy the Premise~\ref{prem_A} and that $\forall \theta\in\Theta,\quad x\mapsto f(x,\theta)$ is strongly convex. Thanks to Lemma~\ref{lem:exist}, for each $\theta\in\Theta$, \eqref{eq:param-optim} has a unique minimizer let us denote $\xsol(\theta)$. 
 
 Now, we want to solve \eqref{eq:param-optim} using an  inertial  methods  for smooth function  forall $\theta\in\Theta.$  
 This inertial methods was introduced in \cite{liang_activity_2017} and named inertial forward-backward methods. The goal was to solve the minimization of the sum of two functions where one is sufficiently smooth and the second is nonsmooth. Here, we suppose that the nonsmooth part is zero.
 The scheme generates the following sequence $\forall k\in\bbN$,
 \begin{algorithm}[htbp]
	\caption{ Inertial Methods}
	\label{eq:inertial-methods}
	\textbf{Parameters:}  $\ak\in[0,\abar], \abar\leq1,  \bk\in[0,\bbar], \bbar\leq1,0<\gak<\frac{2}{L}$\;
	\textbf{Initialization:} $x_{-1}=\xo\in\mathbb{R}^n$, $b_0=a_0=1$,  \;
	\For{$k=0,1,\ldots$}{
		\vspace{-0.25cm}
		\begin{flalign} \tag{IM}\label{BPGalgo} 
			&\begin{aligned} 
				&\yak(\xk,\xkm)=\xk+\ak\pa{\xk-\xkm}; \\
				&\ybk(\xk,\xkm)=\xk+\bk\pa{\xk-\xkm}; \\
				&\xkp=\yak(\xk,\xkm)-\gak\nabla f(\ybk(\xk,\xkm),\theta); \\ 
			\end{aligned}&
		\end{flalign}
	}
\end{algorithm}


This group of inertial algorithms \eqref{eq:inertial-methods} covers standard optimization methods such as:
\begin{itemize}
\item Gradient descent with the choice $\forall k\in\bbN, \ak=\bk=0,$
\item  Heavy-ball method \cite{polyak_methods_1964} with the choice $\forall k\in\bbN, \ak\in [0,\ava]$ and  $\bk=0,$
\item Nesterov's type methods \cite{nesterov_method_1983}  with the choice $\forall k\in\bbN, \ak\in [0,\ava], \bk=\ak \quad \st\quad \ak\to1$ and $\gak\in]0,1/L[$. It is common in this setting to choose $\forall k\in\bbN, \ak=\frac{k-1}{k+3}$ to ensure the  convergence of the iterates.
\end{itemize}
A common  way to properly  analyze this system is to increase the dimension of the problem from $\bbR^n$ to $\bbR^n\times\bbR^n$ by introducing a new variable $\zk=\xkm$. Thus, we get the following iteration 
\begin{equation}\label{eq:nest-system}
\begin{cases}
\xkp=&\yak(\xk,\zk)-\gak\nabla f(\ybk(\xk,\zk),\theta),\\
\zkp=&\xk.
\end{cases}
\end{equation}
Let us define, the following mapping
\begin{equation}\label{eq:mapping}
\forall k\in\bbN, \quad F_k(X,\theta)\eqdef F_k(x,z,\theta)=\begin{pmatrix}
\yak(x,z)-\gak\nabla f(\ybk(x,z),\theta)\\
x
\end{pmatrix},
\end{equation}
where $\transp{X}=\begin{pmatrix}x\\z\end{pmatrix}$. We have the following regularity on  $\seq{F_k
}$ which is stated in the following Lemma. 
\begin{lemma}\label{lem:LipFk} Under our Premise~\ref{prem_A} on the objective function $f$, we have  
\begin{enumerate}
\item $\forall k\in\bbN$, $F_k$ is a $C^1-$smooth function over the space $\bbR^n\times\bbR^n\times\Theta.$  
Moreover, the  Jacobian with respect to $(x,z)$ is
 	\begin{equation}\label{eq:Jac-xz}
	 	J_1F_k(x,z,\theta)\eqdef \begin{pmatrix} (1+\ak)\Id-\gak(1+\bk)\nabla_x^2f(\ybk(x,z),\theta) &-\ak\Id+\gak\bk\nabla_x^2f(\ybk(x,z),\theta) \\ \Id  & 
		0\end{pmatrix},
	\end{equation}
	while the Jacobian with respect to $\theta$ is
	\begin{equation}\label{eq:Jac-theta}
	 	J_2F_k(x,z,\theta)\eqdef \begin{pmatrix} -\gak\nabla_{x\theta}^2f(\ybk(x,z),\theta)  \\  
		0\end{pmatrix}.
	\end{equation}
\item $\forall k\in\bbN$, $F_k$ is also a Lipschitz continuous function with 
\begin{equation}\label{eq:L_F_k}
L_{F_k}= \sqrt{\Ppa{1+\pa{1+\ak}^2+(\gak L)^2\pa{1+\bk}^2}}.
\end{equation}
\end{enumerate}
\end{lemma}
\begin{proof}
We refer the reader to Section \ref{prf:lem:LipFk}.
\end{proof}
We can easily rewrite our inertial scheme Algorithm~\ref{eq:inertial-methods} as follow: $\forall k \in\bbN,\forall\theta\in\Theta$
\begin{equation}\label{eq:rewrite-scheme}
\begin{cases}
& X_{k+1}(\theta)=F_k(X_k,\theta),\\
& \ak\in[0,\abar], \abar\leq1, \bk\in[0,\bbar], \bbar\leq1\qandq 0<\gak<\frac{2}{L}.
\end{cases}
\end{equation}
where we have $\transp{X_k}=\begin{pmatrix}\xk\\\zk\end{pmatrix}$. 
\subsection{Convergence and local analysis}\label{subsec:conv-loc}

The next important questions are, if we Algorithm~\ref{eq:inertial-methods} to solve our optimization problem \eqref{eq:param-optim}, do we have convergence to the unique global minimizer $\xsol(\theta)$ as $k\to\infty$ ? Can we quantify the convergence rate ? In \cite{liang_activity_2017}, the authors  study and analyze the global and the local convergence of the  inertial forward-backwards methods. They provide conditions on the choice the inertial parameter $\seq{\ak},\seq{\bk}$ and  the stepsize $\seq{\gak}$ such that the generated sequence $\seq{\xk}$ converges to $\xsol$. In the next proposition, we specialize their result to our simpler case.   
We pause here,  to make the following premise on the choice of the stepsizes $\seq{\gak}$ and the inertials parameters $\seq{\ak},\seq{\bk}$.
\begin{premise}\label{prem_B} We suppose that there exists a constant $\tau>0$ such that one of the following holds
\begin{itemize}
\item $\forall k\in \bbN,\quad\tau<(1+\ak)-\frac{\gak L}{2}\pa{1+\bk}^2:\ak<\frac{\gak L}{2}b_k$
\item $\forall k\in \bbN,\quad\tau<(1-3\ak)-\frac{\gak L}{2}\pa{1-\bk}^2:\bk\leq\ak \qorq \frac{\gak L}{2}b_k\leq\ak<\bk.$
\end{itemize}
\end{premise}
\begin{remark}Here are two examples of inertial parameters which satisfy Premise~\ref{prem_B}, which can be found in  \cite[Section~5]{liang_activity_2017}. 
\begin{itemize}
\item \textbf{Example 1:} $\forall k\in\bbN, \gak=1/L, \ak=\bk =\sqrt{5}-2-10^{-3},$
\item \textbf{Example 2:} $\forall k\in\bbN, \gak=1/L, \ak=\bk =\frac{k-1}{k+25}.$
\end{itemize}
\end{remark}
The result is stated in the following proposition.
\begin{proposition}[Global convergence of the iterates] \label{pro:conv-iters} Let us  consider, the optimization problem \eqref{eq:param-optim} for all $\theta\in\Theta$ solved with the inertial method \eqref{eq:inertial-methods}. Under the premise~\ref{prem_A}, premise~\ref{prem_B} and the strong convexity hypothesis with respect to $x$, we have that the generated sequence $\seq{\xk(\theta)}$ has finite length and the sequence $\seq{\xk(\theta)}$ converges to $\xsol(\theta)$. 
\end{proposition}
\begin{proof}
We refer the interested reader to \cite[Section~A]{liang_activity_2017} for a detailed proof of this proposition.
\end{proof}
\begin{remark}\label{rem:conv-iters}$~$
\begin{itemize}
\item This proposition is a particular case of \cite[Theorem~4]{liang_activity_2017}, here we take the nonsmooth part to be zero.
\item Since the inertial methods  can be rewritten as the scheme \eqref{eq:rewrite-scheme} in $\bbR^n\times\bbR^n$, Proposition~\ref{pro:conv-iters} entails that the sequence $\seq{X_k(\theta)}$ converges to a unique point $X^{\star}(\theta )$ where $\transp{X^{\star}(\theta )}\eqdef\begin{pmatrix}\xsol(\theta)\\\xsol(\theta)\end{pmatrix}.$ 
\end{itemize}
\end{remark}

\paragraph*{Local behavior of the IM algorithm} $~$

Let us recall that thanks to Lemma~\ref{lem:LipFk}, $\forall k\in\bbN,$ the function $F_k$ defined in \eqref{eq:mapping} is $C^1-$smooth on $\bbR^n\times\bbR^n\times\Theta$ and set $\forall k\in\bbN,$
\begin{align*}\label{eq:Jsol}
 M_k\eqdef& J_1F_k(\xsol,\xsol,\theta)\\
=& \begin{pmatrix} (1+\ak)\Id-\gak(1+\bk)\nabla_x^2f(\xsol,\theta) &-\ak\Id+\gak\bk\nabla_x^2f(\xsol,\theta) \\ \Id  & 
		0\end{pmatrix}.\numberthis
\end{align*}
For any $a,b\in[0,1]$ and $\gamma\in]0,2/L[$, we define 
\begin{align}\label{eq:def-mat}
M\eqdef& \begin{pmatrix} (1+a)\Id-\gamma(1+b)\nabla_x^2f(\xsol,\theta) &-a\Id+\gamma b\nabla_x^2f(\xsol,\theta) \\ \Id  & 
		0\end{pmatrix},\\
	=& \begin{pmatrix} (a-b)\Id+(1+b)G_{\theta} &-(a-b)\Id-bG_{\theta} \\ \Id  & 
		0\end{pmatrix},
\end{align}
where $G_{\theta}\eqdef \Id-\gamma\nabla_x^2f(\xsol,\theta).$

 Let us observe that for $\gamma\in]0,2/L[, G_{\theta}$ has eigenvalues in $]-1,1[$ and if $\gamma\in[0,1/L[, G_{\theta}$ has eigenvalues in $[0,1[$. Therefore, let us denote by $\etam_{\theta}$ and $\etaM_{\theta}$ the smallest and the largest eigenvalue of $G_{\theta}$.

We will need the following premise to pursue our study. 
\begin{premise}\label{prem_C} $~$
\begin{enumerate}[label=(C.\arabic*)]
 \item \label{prem_C1} There exists $a,b\in [0,1]$ and $\gamma\in]0,\frac{2}{L}[$ such that the sequences $\ak \to a$, $\bk \to b$ and $\gak\to \gamma.$ 
 \item\label{prem_C2} Given any limits point $a,b\in[0,1[^2$ of the sequence $\seq{\ak},\seq{\bk}$ respectively the following holds:  
 \[
 \frac{2\pa{b-a}-1}{1+2b}<\etam_{\theta}.
 \]
\end{enumerate}
\end{premise}
\begin{remark} Premise~\ref{prem_C1} suppose that  the sequence of inertial parameters $\seq{\ak}, \seq{\bk}$ and stepsize $\seq{\gak}$  have  limits, while Premise~\ref{prem_C2} is an hypothesis on the link between the limits $a,b$ and the lowest eigenvalue of $G_{\theta}$. This premise is crucial to get the linear convergence rates.  
\end{remark}

We have the following proposition.
\begin{proposition}[Local linear convergence]\label{pro:local-lin} Let us consider the optimization problem \eqref{eq:param-optim} for all $\theta\in\Theta$ solved with the inertial method \eqref{eq:inertial-methods}. Under the premise~\ref{prem_A}, \ref{prem_B}, \ref{prem_C} and the strong convexity hypothesis with respect to $x$, we have  that

\begin{enumerate}
\item\label{pro:local-lin1}  the sequence of matrices $\seq{M_k}$ converges to the matrice $M$ defined in \eqref{eq:def-mat},

\item \label{pro:local-lin2} the spectral radius of $M$ is such that $\rho(M)<1,$

\item\label{pro:local-lin3} for any $\rho\in[\rho(M),1[,$ there exists $K>0$ large enough and a constant $C>0$ such that for all $k\geq K,$ it holds 
\begin{centerbox}{black}{}y
\begin{equation}\label{eq:local-lin}
	\normm{X_k(\theta)-X^{\star}(\theta)}\leq C\rho^{k-K}\normm{X_K(\theta)-X^{\star}(\theta)}.
\end{equation}
\end{centerbox}

\end{enumerate}
\end{proposition}
\begin{proof}$~$
The proof can be found in Section~\ref{prf:pro:local-lin}.
\end{proof}

\section{Differentiation of  inertial methods }\label{sec:ad-ima}
In this section, our main purpose is to differentiate the inertial methods for solving the parametric optimization problem \eqref{eq:param-optim} for any $\theta\in\Theta\subseteq\bbR^m$ is an open set. We pause here to recall some basics notion about differentiation. 

We recall that ``automatic'' differentiation is a method to compute the derivatives of an iterative process with respect to some parameter. Let us start by properly define what we consider to be an infinite iterative process. We refer the reader to \cite[Section~2]{beck1994} for a detailed exposure. 
\begin{definition}\label{def:codelist}An infinite iterative process is the given of the following triplet  $\calJ=\Ppa{\seq{\Phi_k},\Theta,\xo}$ where $\Theta\subseteq\bbR^m$ and 
we have 
\begin{enumerate}
\item $\xo:\Theta\to\bbR^n$  and $\Phi_k:\bbR^{n}\times\Theta\to\bbR^n,\forall k\in\bbN$ are continuously differentiable functions,
\item we have that the following iteration 
	\begin{equation}\label{eq:code-list} 
			\xkp(\theta)=\Phi_k(\xk(\theta),\theta),\quad \forall k\in\bbN,
	\end{equation}
	exist. 
\end{enumerate}
Moreover, if the sequence of function $\xk(\theta) \to \avx(\theta)$ then $\avx$ is called \textit{limit of the infinite iterative process}.
\end{definition} 
If we suppose that the sequence $\seq{\xk(\cdot)}$ is differentiable, the following intuitive question  for the consistency of the differentiation of  the infinite iterative process is: if the sequence $\seq{\xk(\theta)}$  converges to $\avx(\theta)$, do we have that the sequence of the derivatives $\seq{\partial_{\theta}\xk(\theta)}$ converge also  to the derivative of the limit $\partial_{\theta}\avx(\theta)$ \ie

\textit{Do we have that  $\xk(\theta) \to \avx(\theta)$ implies that $\partial_{\theta}\xk(\theta) \to \partial_{\theta}\avx(\theta)$?}  This is not a simple question and depends on the properties of the iteration map $\Phi_k$. Thus, we need the following definition.
\begin{definition}[Derivative stable]\label{def:der-sta} 
A sequence of differentiable function $\seq{\xk(\cdot)}$ is called \textit{derivative-stable} if and only if there exists a limit function $\avx(\theta)$ such that the following holds:
\begin{equation}\label{eq:der-sta}
\xk(\theta) \to \avx(\theta) \Longrightarrow \partial_{\theta}\xk(\theta) \to \partial_{\theta}\avx(\theta).
\end{equation}
\end{definition}

In our setting, we want to investigate whether or not the sequence of inertial methods is derivative stable.   We recall that the sequence is given by the following iterations 
\begin{equation}\label{eq:rewrite-scheme2}
\begin{cases}
& X_{k+1}(\theta)=F_k(X_k,\theta),\quad \forall k\in\bbN,\\
& \ak\in[0,\abar], \abar\leq1, \bk\in[0,\bbar], \bbar\leq1\qandq 0<\gak<\frac{2}{L}.
\end{cases}
\end{equation}
where we have $\transp{X_k}=\begin{pmatrix}\xk\\\zk\end{pmatrix}=\begin{pmatrix}\xk\\\xkm\end{pmatrix}\in\bbR^{2n}$. 

Before our investigation, we have the following Lemma which proves that any sequence generated by Algorithm~\ref{eq:inertial-methods} in the form \eqref{eq:rewrite-scheme2} is an infinite iterative process.
\begin{lemma}\label{lem:def-inf} Let us consider the  inertial methods \eqref{eq:rewrite-scheme2}. Under Premise~\ref{prem_A}, if $X_0:\Theta\to\bbR^{2n}$ is a continuously differentiable function with respect to the parameter $\theta$ then the triplet $\calA\eqdef\Ppa{\seq{F_k},\Theta,X_0}$ is an infinite  iterative process according to Definition~\ref{def:codelist}.
\end{lemma}
\begin{remark}\label{rmk:def-inf} This Lemma may seem spare as first sight but it is important for the analysis of any sequence  generated by a computer program. Mainly, in which sense we consider the sequence of functions that are generated. In our setting, we properly defined what is called an infinite iterative process and Lemma~\ref{lem:def-inf} show that  the inertial methods that we consider in this work, under suitable premises on the the objective function, is an infinite iterative process. Finally, any analysis of functions generated by a computer program should follow this part in view to be  more precise and mathematically correct.
\end{remark}

\begin{proof}
The proof is straightforward. Indeed by hypothesis,  the set of parameter $\Theta$ is an open subset of $\bbR^m$, $X_0$ is continuously differentiable function  with respect to the parameter $\theta$ to the space $\bbR^n\times\bbR^n$ and Lemma~\ref{lem:LipFk} ensures that $\forall k\in\bbR^n, F_k$ is continuously differentiable over the space $\bbR^n\times\bbR^n\times\Theta$.  Algorithm~\ref{eq:inertial-methods} in the form \eqref{eq:rewrite-scheme2} ensures that the sequence of functions generated is in the form \eqref{eq:code-list}. We finally conclude  that  the triplet $\calA$ is an infinite  iterative process according to Definition~\ref{def:codelist}.
\end{proof}

Let us denote  $\forall k\in\bbN, \partial_{\theta}X_k(\theta)\in \bbR^{2n\times m}$  to be  the Jacobian of $X_k(\theta)$ with respect to the parameter $\theta$. By the rule of derivation of composed functions, we have 
\begin{equation}\label{eq:pdcf-rule-in}
 \forall k\in\bbN, \quad \partial_{\theta}X_{k+1}(\theta)=J_1F_k(X_k(\theta),\theta)\partial_{\theta}X_k(\theta)+J_2F_k(X_k(\theta),\theta),
\end{equation}
where we have from \eqref{eq:Jac-xz} and \eqref{eq:Jac-theta} that 
\[
	 	J_1F_k(X_k\pa{\theta},\theta)= \begin{pmatrix} (1+\ak)\Id-\gak(1+\bk)\nabla_x^2f(\ybk(\xk,\xkm),\theta) &-\ak\Id+\gak\bk\nabla_x^2f(\ybk(\xk,\xkm),\theta) \\ \Id  & 
		0\end{pmatrix},
\]
and 
\[
		 	J_2F_k(X_k\pa{\theta},\theta)= \begin{pmatrix} -\gak\nabla_{x\theta}^2f(\ybk(\xk,\xkm),\theta)  \\  
		0\end{pmatrix}.
\]

\begin{remark}\label{rmk:inert-deriv}
 We  can observe from \eqref{eq:pdcf-rule-in} that the sequence of derivatives $\seq{\partial_{\theta}X_{k}(\theta)}$ is linear in $\partial_{\theta}X_{k}(\theta)$ with a perturbation  term which is represented by the function  $J_2F_k(X_k\pa{\theta},\theta)$. This simply  means  that  after  the differentiation  procedure of  inertial methods  the sequence of derivatives has no oscillations in term of $\partial_{\theta}X_{k}(\theta)$. As we may observe later in the numerical experiments. 
\end{remark}

\subsection{Convergence result of the derivative of the IM algorithm}	\label{subsec:gim-stable}
The following is our main convergence result of the derivative of the  inertial methods.
\begin{theorem}\label{thm:gim-stable} Let us consider the  inertial method Algorithm~\ref{eq:inertial-methods} in the form \eqref{eq:rewrite-scheme2} for minimizing the parametric optimization problem \eqref{eq:param-optim}, for each $ \theta\in\Theta$ . Under Premise~\ref{prem_A} on the objective function $f$ with the additional hypothesis that $f$ is strongly convex with respect to $x$, Premise \ref{prem_B} and \ref{prem_C} on the parameters of the Algorithm~\ref{eq:inertial-methods}  and that $X_0:\Theta\to\bbR^{2n}$ is a continuously differentiable function with respect to the parameters $\theta$ then, the following holds:
\begin{enumerate}
\item\label{thm:gim-stable1} the sequence of derivatives $\seq{\partial_{\theta}X_{k}\pa{\theta}}$ converges pointwise to the derivative of  $\partial_{\theta}X^{\star}\pa{\theta},$ 
\item\label{thm:gim-stable2} moreover, we have the following explicit  formula:  $\forall\theta\in\Theta,$
\begin{centerbox}{black}{}y
\small{
\begin{equation}\label{eq:expli-form}
 \hspace{-0.8cm}\partial_{\theta}X^{\star}\pa{\theta}=\begin{pmatrix} -a\Id+\gamma(1+b)\nabla_x^2f(\xsol(\theta),\theta) &a\Id-\gamma b\nabla_x^2f(\xsol\pa{\theta},\theta) \\ -\Id  & 
		\Id\end{pmatrix}^{-1}\hspace{-0.2cm}\begin{pmatrix} -\gamma\nabla_{x\theta}^2f(\xsol,\theta)  \\  
		0\end{pmatrix}.
\end{equation}}
\end{centerbox}

\end{enumerate}
\end{theorem}
\begin{remark} Clearly,  the claim~\ref{thm:gim-stable1} of this  theorem tell us that the  any sequence generated by our inertial methods Algorithm~\ref{eq:inertial-methods} in the form \eqref{eq:rewrite-scheme2}  under our hypothesis is derivative stable. Moreover, we have an explicit formula of the corresponding derivative in claim~\ref{thm:gim-stable2} and given by \eqref{eq:expli-form}.
\end{remark}
\begin{proof}$~$ 
We refer the reader to Section~\ref{prf:thm:gim-stable}.
\end{proof}

\subsection{Convergence rates of the derivative }\label{subsec:conv-ad-im}
In this section, we try to answer the question at which rates the convergence in Theorem~\ref{thm:gim-stable} occurs? This question may seem not important at first sight but is crucial to examine the behaviour of the derivative sequence. Clearly, we want to quantify this rates of convergence. We are ready now to state our main result on the convergence rate of the derivatives of the proposed inertial method Algorithm~\ref{eq:inertial-methods}.
\begin{theorem}[Convergence rate]\label{thm:conv-rat-der} Let us consider the inertial method Algorithm~\ref{eq:inertial-methods} in the form \eqref{eq:rewrite-scheme2} for minimizing the parametric optimization problem \eqref{eq:param-optim}, $ \theta\in\Theta$ . 

Under Premise~\ref{prem_A} on the objective function $f$, with the additional hypothesis that $f$ is strongly convex with respect to $x$, Premise \ref{prem_B} and \ref{prem_C} on the parameters of the Algorithm~\ref{eq:inertial-methods}  and that $X_0:\Theta\to\bbR^{2n}$ is a continuously differentiable function with respect to the parameters $\theta$ then for $\eps>0$ (small enough) their exists $K$ large enough such that:
\begin{centerbox}{black}{}y

\begin{multline}\label{thm:conv-rat-der}
\hspace{-1cm}\forall k\geq K, \normm{\partial_{\theta}X_{k+1}(\theta)-\partial_{\theta}X^{\star}(\theta)}\leq 2\rho\pa{M}\normm{\partial_{\theta}X_k(\theta)-\partial_{\theta}X^{\star}(\theta)}+\eps\Ppa{2+\normm{\partial_{\theta}X_k(\theta)-\partial_{\theta}X^{\star}(\theta)}}.
\end{multline}

\end{centerbox}
\end{theorem}

\begin{remark}\label{rem:conv-rat-der}$~$
	\begin{itemize}
		\item Let us emphasize  that this result is very new to the best of our knowledge compared  to other works which quantify the convergence rate under an additive hypothesis such that Lipschitz continuity of the second order derivatives of the  objective  $f$. This additive hypothesis is quite restrictive in practice and does not apply to  lot of standard problems. Let us consider for instance this commonly used function in the literature of phase retrieval like \cite{bolte_first_2017,godeme2024stable,mukkamala_global_2020}  
		\[
			\psi(x)=\frac{1}{4}\normm{x}^4+\frac{1}{2}\normm{x}^2,
		\] 
		whose second order derivatives can be written as 
		\[
		\nabla_x^2\psi(x)=\Ppa{\normm{x}^2+1}+2x\transp{x}\preceq \Ppa{3\normm{x}^2+1}\Id
		\] 
		and thus not globally bounded from above. 
		\item We show ``almost'' the same result without this hypothesis. Therefore, our result has the advantage to be applicable to a broader set of problems.  Indeed, \eqref{thm:conv-rat-der} shows that the local linear convergence occurs in the long term regime of the sequence with a small additive error term. We observe that this error terms vanish as the number of iterations increases and tends to infinity. \textit{Moreover, we would like to mention to the most curious reader that this result goes far beyond the particular case of  inertial methods that we study in this paper, we think that  it holds true for any infinite iterative process which satisfies the standard Premise~\ref{prem_apend} since in the proof we never used the explicit form of the quantity involved.} 
	\end{itemize}
	 
\end{remark}

\begin{proof}$~$

The proof of this Theorem can be found in Section~\ref{prf:thm:conv-rat-der}.
\end{proof}

We get the following corollary which represent  the convergence rate and can be useful in applications. 
\begin{corollary}
Let us  define 
\[
\tau\eqdef 2\rho(M),\quad  \forall k\in\bbN, \quad g(X_k)\eqdef 2+\normm{\partial_{\theta}X_k(\theta)-\partial_{\theta}X^{\star}(\theta)}.
\]
Under the same hypothesis  as Theorem~\ref{thm:conv-rat-der} then for $\eps>0$ (small enough) their exists $K$ large enough such that:
\begin{centerbox}{black}{}y
\begin{multline}\label{cor:conv-rat-der}
\forall k\geq K,  \normm{\partial_{\theta}X_{k+1}(\theta)-\partial_{\theta}X^{\star}(\theta)}\leq \tau^{k+1-K}\normm{\partial_{\theta}X_K(\theta)-\partial_{\theta} X^{\star}(\theta)}+\eps\sum_{i=K}^k\tau^{i-K}g(X_i).
\end{multline}
\end{centerbox}
\end{corollary}
\begin{remark}$~$
\begin{itemize}

\item Let us  mention that our convergence rate is different from existing convergence rate in the literature of differentiation.
In fact, we prove a local convergence while most work try to prove a global convergence. Moreover,  our rate is $\tau$ with an vanishing error term.  It is better  to bound  it (strictly) with $1$.  Work arounds like \cite{mehmood19,iutzeler24} have a convergence rate  usually of the form  $\max\Ba{\rho(M),q_x}$ where $q_x$ is said to be the convergence rate of the sequence $\seq{\xk(\theta)}$ and without the error term in \eqref{cor:conv-rat-der}  which is quite important to notice. 
\item We  want to answer the question why our convergence rates is so different from previous work ? Our results rely on a careful analysis of the problem without bounding  the second order derivative from the beginning. We mean by ``careful analysis'' of the problem take into account all the premises of the problem without adding additional hypothesis which in this case is superfluous.
\end{itemize}
\end{remark}
\begin{proof}$~$This result is a  consequence of the previous Theorem by summing from  $K$ to $k$. 
\end{proof}

\section{Numerical Experiments}\label{sec_numexp}
In this section, we highlight our results by examining the numerical behavior of the iterates and derivatives for an application. 
All the experiments were carry out using Matlab software (\cite{matlab22}). 

We consider that we want solve the following linear inverse problem 
\begin{equation}\label{eq:inv-probl}\tag{$\calP_{\mathrm{inv}}$}
\begin {cases}
\text{Recover a real vector or signal $\avx(\theta)\in \bbR^n,$  from }\\
y(\theta)= A\avx(\theta)\in \bbR^m \qwhereq \avx(\theta)=\frac{1}{2}\tilde{x}\theta^2, 
\end{cases}
\end{equation}
where we have $A\in \bbR^{m\times n}$ is the measurements matrix whose rows $(a_r)_{r\in[m]}$ are the measurements vectors.  Our parameter   $\theta\in\bbR$  can be seen as a scaling coefficient of the fixed signal  $\tilde{x}\in\bbR^n$ which is also a fixed vector.

\subsection{Solving using Least-squares or linear regression}\label{sec:least-square}
To solve \eqref{eq:inv-probl}, we consider the following parametric optimization problem  
\begin{equation}\label{eq:param-ls}
\forall \theta\in\bbR, \quad \min_{x\in\bbR^n}f(x,\theta)=\frac{1}{2}\normm{y(\theta)-Ax}^2.
\end{equation}
 For all $ x\in\bbR^n, \theta\in\bbR$, we have 
\begin{align}\label{eq:derivative}
\nabla_x f(x,\theta)=-\transp{A}\Ppa{y(\theta)-Ax},\quad\nabla^2_x f(x,\theta)=\transp{A}A,\qandq\nabla^2_{x\theta} f(x,\theta)=-\pa{\transp{A}A}\tilde{x}\theta.
\end{align}
From \eqref{eq:derivative}, we have that $\forall \theta \in \Theta$ 
\begin{align}
\normm{\nabla_x f(x,\theta)-\nabla_x f(z,\theta)} \leq \normm{A} \normm{x-y},
\end{align}
thus the Lipschitz-coefficient is given by $L=\normm{A}.$
 
We recall that our goal is to solve \ref{eq:param-ls} using the inertial scheme Algorithm~\ref{eq:inertial-methods} and then differential through the latter. Given any smooth function of $\theta$ as an initial point: $\forall \theta\in \bbR, \xo(\theta)\in\bbR^m$, we define $X_0(\theta)=\begin{pmatrix} \xo(\theta)\\ \xo(\theta)\end{pmatrix}$ then the automatic differentiation produces the following sequence:
\begin{multline}
 \forall k\in\bbN, \quad  \hspace{-0.2cm} \partial_{\theta}X_{k+1}(\theta)= \begin{pmatrix} (1+\ak)\Id-\gak(1+\bk)\transp{A}A &-\ak\Id+\gak\bk\transp{A}A \\
		 \Id  & 0\end{pmatrix}\partial_{\theta}X_k(\theta)\\+\begin{pmatrix} \gak\pa{\transp{A}A}\tilde{x}\theta  \\  
		0\end{pmatrix}.
\end{multline}
Thanks to  the explicit formula \eqref{eq:expli-form}, we get 
  \begin{equation}
 \hspace{-0.8cm}\partial_{\theta}X^{\star}\pa{\theta}=\begin{pmatrix} -a\Id+\gamma(1+b)\transp{A}A &a\Id-\gamma b\transp{A}A \\ -\Id  & 
		\Id\end{pmatrix}^{-1}\hspace{-0.2cm}\begin{pmatrix} \gamma\pa{\transp{A}A}\tilde{x}\theta  \\  
		0\end{pmatrix}.
\end{equation}
\begin{remark} Although this Ordinary least-square case looks very basic, it is important  and as highlights, the computations are done handily so that we can observe the behaviour of our  differentiation procedure throughout.   
\end{remark}
\vspace{-0.8cm}
\paragraph{Experience setup} 
In our numerical simulations, we suppose that $\forall r \in [m],  a_r \stackrel{\footnotesize{\text{\iid}}}{\sim}\calN(0,\Id)$ and the Lipschitz coefficient of the gradient $L\approx\sqrt{m}$. Moreover, we choose for simplicity that  $\xo\in\bbR^n,$ a constant function of $\theta$ which implies that $\partial_{\theta}\xo(\theta)=0$. For each simulation, we run the inertial scheme for $400$ iterations and  we compare  the results with the solution $\xsol(\theta)$ and his derivatives $\partial_{\theta}\xsol(\theta)$. Finally,  we select the parameters of our algorithm  to get two different scenario: 
\begin{itemize}
\item \textbf{Case 1:} $\gak\equiv\frac{1}{L-2/k}, \ak=\bk=\frac{k-1}{k+20}.$ We easily get that $a=b=1$ and $\gamma=1/L.$
\item \textbf{Case 2:} $\gak\equiv\frac{1}{L-2/k}, \ak=\bk=0,$ with  $a=b=0$ and $\gamma=1/L.$
\end{itemize}
\begin{figure}[htbp]
\centering
\includegraphics[trim={0cm 8.5cm 0cm 8.5cm},clip,width=1\textwidth]{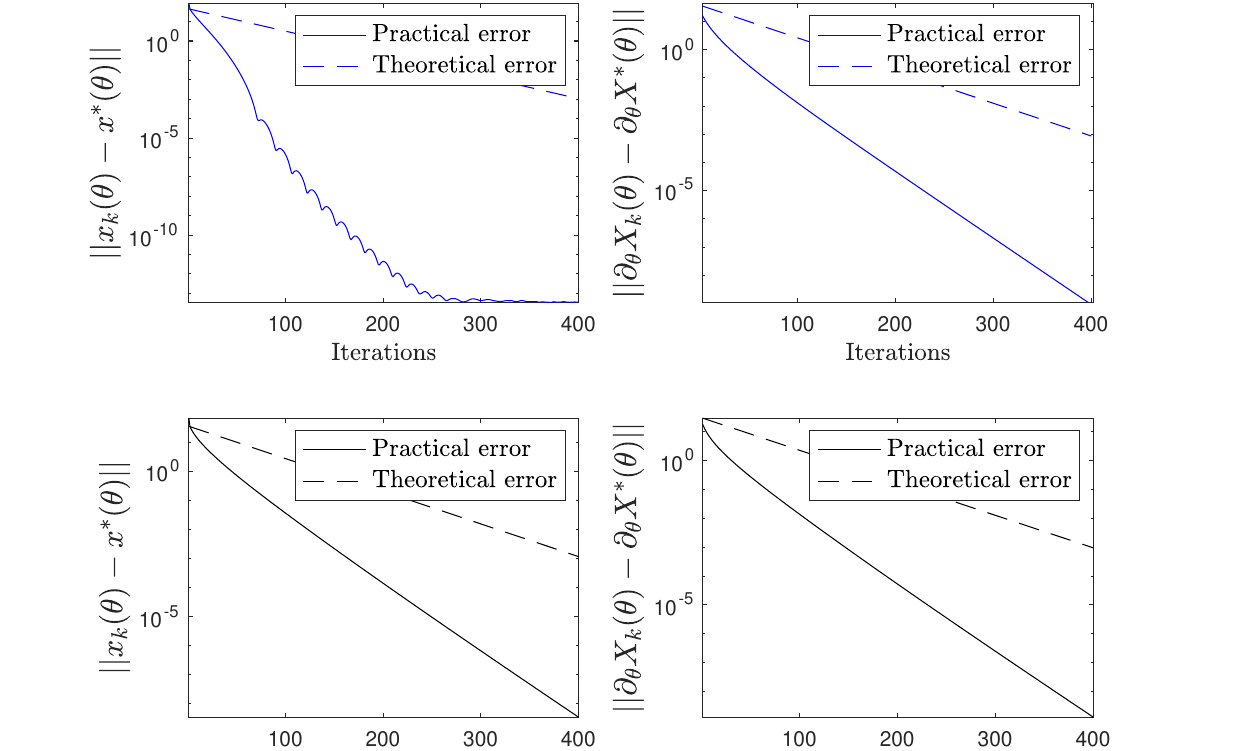}
 \caption{Automatic differentiation of the inertial methods (Case 1=inertial method) and (Case 2=  Gradient  descent).}
\label{fig: linea-inv-vstep}
\end{figure} 

\paragraph{Observations }  In Figure~\ref{fig: linea-inv-vstep}, we  displayed the result of our numerical experiments for solving the problem \eqref{eq:inv-probl}  with the least-square formulation \eqref{eq:param-ls}. The first line in blue is the result when we use a standard inertial method. On the left hand side, we have the  error of the of the inertial schemes with respect to the number of iterations and in dashed line we have the Theoretical error which is of course not optimistic. We can  observe with the experiments oscillations which characterize inertial methods. On the right hand side, the error of the differentiation methods with respect to the number of iterations. In dashed line, the theoretical linear convergence without the error term and in plain line the observed error.  As predicted by Theorem~\ref{thm:conv-rat-der}, the linear convergence does not occur from the start as described by many automatic differentiation daily users. We have a small regime without the linear convergence. In this regime, the error term is large which impact the convergence of the differentiation scheme.  We can see that the error is not linear. But after a few iterations the differentiation procedure enters a linear convergence regime. And  this is what we have highlighted in our Theorem~\ref{thm:conv-rat-der}. Moreover, as predicted in Remark~\ref{rmk:inert-deriv}, after differentiation they are no more oscillations, we have a linear convergence after few iterations. On the second scenario, we solved the problem using the gradient descent and plotted the results in black. We  can observe in contrast to the previous case that we  do not have oscillations which also describe this scheme. We have made the same experiments, and on the left hand side, after few iterations, the sequence enters a linear convergence rate. On the right side, the differentiation of the gradient descent show the same behaviour as the  sequence generated by the gradient descent. Theses numerical experiments confirm all our theoretical results.  

\subsection{The Rockefellar's log-exponential}\label{sec:logexp}
Let us start by briefly introduce what is called the log-exponential function  \ie 
\begin{equation}
\forall x \in \bbR^m,\quad  \logexp(x)\eqdef \log\BPa{e^{x_1}+\cdots+e^{x_m}}.
\end{equation} 
According to \cite[Example~2.16]{rockafellar_variational_1998}, the log-exponential function is only a convex function. However, in the context of solving \eqref{eq:inv-probl}, we consider  the following parametric optimization problem  
\begin{equation}\label{eq:param-logexp}
\forall \theta\in\bbR, \quad \min_{x\in\bbR^n}f(x,\theta)=\frac{1}{2}\BPa{\logexp\Ppa{y(\theta)-Ax}}^2,
\end{equation}
To simplify the notations, we denote by $\sigma(y)\eqdef\sum_{j=1}^me^{y_j}, \forall y\in \bbR^m$. We rewrite \eqref{eq:param-logexp} as
\[
\forall \theta\in\bbR,\quad f(x,\theta)=\frac{1}{2}\BPa{\log\pa{\sigma\Ppa{y(\theta)-Ax}}}^2.
\]
We have that,  for all $ x\in\bbR^n, \theta\in\bbR$  we have that the gradient  of $f$ is given
\begin{equation}\label{eq:gradient-formula}
\nabla_x f(x,\theta)=\frac{-\logexp\Ppa{y(\theta)-Ax}}{\sigma\Ppa{y(\theta)-Ax}}\transp{A}\nabla\sigma\Ppa{y(\theta)-Ax}.
\end{equation}
 The second order derivatives of interest are respectively
\begin{multline}\label{eq:hes_xx}
\nabla^2_xf(x,\theta)=\frac{\logexp\Ppa{y(\theta)-Ax}-1}{\sigma\Ppa{y(\theta)-Ax}^2}\transp{\BPa{\nabla\sigma\Ppa{y(\theta)-Ax}A}}\nabla\sigma\pa{y(\theta)-Ax}A\\-\frac{\logexp\Ppa{y(\theta)-Ax}}{\sigma\pa{y(\theta)-Ax}}\transp{A}\BPa{\diag{\nabla\sigma\pa{y(\theta)-Ax}}}A,
\end{multline}
 and 
\begin{multline}\label{eq:hes_xt}
\nabla^2_{x\theta}f(x,\theta)=\frac{\logexp\Ppa{y(\theta)-Ax}-1}{\sigma\Ppa{y(\theta)-Ax}^2}\transp{\BPa{\nabla\sigma\Ppa{y(\theta)-Ax}A}}\nabla\sigma\pa{y(\theta)-Ax}\tilde{x}\theta\\-\frac{\logexp\Ppa{y(\theta)-Ax}}{\sigma\pa{y(\theta)-Ax}}\transp{A}\BPa{\diag{\nabla\sigma\pa{y(\theta)-Ax}}}\tilde{x}\theta.
\end{multline}
where 
$
\quad\nabla\sigma(y)=\begin{pmatrix}e^{y_1}\\\vdots\\e^{y_m}\end{pmatrix}.$
As previously, given any smooth function of $\theta$ as an initial point:\\
 $\forall \theta\in \bbR, \xo(\theta)\in\bbR^m$, we define $X_0(\theta)=\begin{pmatrix} \xo(\theta)\\ \xo(\theta)\end{pmatrix}$ then the derivative sequence is the following: $ \forall k\in\bbN, $
\[
\ybk(\xk,\xkm)=\xk\pa{\theta}+\bk\pa{\xk\pa{\theta}-\xkm\pa{\theta}},
\]
\begin{multline}
 \hspace{-0.2cm} \partial_{\theta}X_{k+1}(\theta)= \begin{pmatrix} (1+\ak)\Id-\gak(1+\bk)\nabla^2_xf\pa{\ybk\pa{\xk,\xkm},\theta} &-\ak\Id+\gak\bk\nabla^2_xf\pa{\ybk\pa{\xk,\xkm},\theta} \\
		 \Id  & 0\end{pmatrix}\\\partial_{\theta}X_k(\theta)+\begin{pmatrix}-\gak\nabla_{x\theta}^2f\pa{\ybk\pa{\xk,\xkm},\theta} \\  
		0\end{pmatrix}.
\end{multline}
Where the different formulas are in \eqref{eq:hes_xx} and \eqref{eq:hes_xt}. Finally, we have
 \begin{equation}
 \hspace{-0.8cm}\partial_{\theta}X^{\star}\pa{\theta}=\begin{pmatrix} -a\Id+\gamma(1+b)\nabla_x^2f(\xsol(\theta),\theta) &a\Id-\gamma b\nabla_x^2f(\xsol\pa{\theta},\theta) \\ -\Id  & 
		\Id\end{pmatrix}^{-1}\hspace{-0.2cm}\begin{pmatrix} -\gamma\nabla_{x\theta}^2f(\xsol,\theta)  \\  
		0\end{pmatrix}.
\end{equation}

\paragraph{Observations }   We have re-introduced the log-exponential function in view to solve inverse problem as an alternative to logistic regression commonly used in the learning literature. In this work, we would like to do the same thing as the previous  example and  without any mathematical background and we encounter technical  difficulties.  Indeed,  we do not have strong convexity or unicity of the minimizer. Therefore, our theoretical results are not directly applicable. However, we observe that  for any $\theta \in \Theta$ at $\xbar(\theta)$
\begin{align*}
\nabla f(\xbar(\theta))=0,\qandq  \nabla_x^2 f(\xbar(\theta),\theta)=-\frac{e}{m}\transp{A}A\preceq0.
\end{align*}
This  suggests that $\xbar(\theta)$ is a particular critical point of $f(x,\theta)$. We leave the study of the properties of this function for solving inverse problem as a future work.

	\begin{appendices}
\section{Proofs for the convergence of the sequence}\label{sec:pr-conv-seq}

\subsection{Intermediate result}\label{sec:int-conv}
We have the following Lemma.
\begin{lemma}\label{lem:inter-conv} Let us define the following matrices, for $k\in\bbN,$ 
\begin{align*}
&M_{k,1}\eqdef[(1+b)\pa{G^k_{\theta}-G_{\theta}},-b\pa{G^k_{\theta}-G_{\theta}}],\\
&M_{k,2}\eqdef[\Ppa{\pa{\ak-\bk}-\pa{a-b}}\Id+\pa{\bk-b}G^k_{\theta},-\Ppa{\pa{\ak-\bk}-\pa{a-b}}\Id-\pa{\bk-b}G^k_{\theta}]. 
\end{align*}
Then there exists $K$ large enough such that
\[
\normm{M_{k,2}\Ppa{X_{k}(\theta)-X^{\star}(\theta)}}=\normm{M_{k,1}\Ppa{X_{k}(\theta)-X^{\star}(\theta)}}=o\Ppa{\normm{X_{k}(\theta)-X^{\star}(\theta)}}=0.
\]
\end{lemma}
\begin{proof} The proof of Lemma~\ref{lem:inter-conv}  can be found in a more general form in \cite[Proposition~27]{liang_activity_2017}.
\end{proof}
\subsection{Proof of Lemma~\ref{lem:LipFk}}\label{prf:lem:LipFk}
\begin{proof}$~$ 
The proof of the first claim comes directly from the Premise~\ref{prem_A1}, the definition of the scheme \eqref{eq:mapping} and straightforward differentiation.

 For the second claim,  $\forall k\in\bbN,$   for any  $(X_1,\theta_1)$ and $(X_2,\theta_2)$ in $\bbR^n\times\bbR^n\times\Theta,$ we have that 
 \begin{align*}
 F_k(X_1,\theta_1)-F_k(X_2,\theta_2)&=F_k(x_1,z_1,\theta_1)-F_k(x_2,z_2,\theta_2)\\
 &=\begin{pmatrix}
\Ppa{\yak(x_1,z_1)-\yak(x_2,z_2)}-\gak\Ppa{\nabla f(\ybk(x_1,z_1),\theta_1)-\nabla f(\ybk(x_2,z_2),\theta_2)}\\
x_1-x_2
\end{pmatrix}
 \end{align*}
therefore by triangular inequality we get,
\begin{multline*}
\normm{F_k(X_1,\theta_1)-F_k(X_2,\theta_2)}^2\leq\normm{x_1-x_2}^2+ \normm{\yak(x_1,z_1)-\yak(x_2,z_2)}^2\\+\gak^2\normm{\nabla f(\ybk(x_1,z_1),\theta_1)-\nabla f(\ybk(x_2,z_2),\theta_2)}^2
\end{multline*}
On one hand we have  that 
\[
\normm{\yak(x_1,z_1)-\yak(x_2,z_2)}^2\leq\pa{1+\ak}^2\normm{x_1-x_2}^2+\ak^2\normm{z_1-z_2}^2,
\]
On the other hand, we use the fact that $\nabla_x f$ is $L-$Lipschitz continuous and the positivity of the norm to obtain that
\[
\normm{\nabla f(\ybk(x_1,z_1),\theta
_1)-\nabla f(\ybk(x_2,z_2),\theta_2)}^2\leq L^2\Ppa{\pa{1+\bk}^2\normm{x_1-x_2}^2+\bk^2\normm{z_1-z_2}^2+\normm{\theta_1-\theta_2}^2}. 
\]
We sum each bounds and we arrive at  
\begin{multline*}
\normm{F_k(X_1,\theta_1)-F_k(X_2,\theta_2)}^2\leq \normm{x_1-x_2}^2\Ppa{1+\pa{1+\ak}^2+(\gak L)^2\pa{1+\bk}^2}+\\\normm{z_1-z_2}^2\Ppa{\ak^2
+\bk^2\pa{L\gak}^2}+\normm{\theta_1-\theta_2}^2,
\end{multline*}
hence we obtain,
\begin{equation*}
\normm{F_k(X_1,\theta)-F_k(X_2,\theta)}\leq\sqrt{\Ppa{1+\pa{1+\ak}^2+(\gak L)^2\pa{1+\bk}^2}}\sqrt{\normm{X_1-X_2}^2+\normm{\theta_1-\theta_2}^2}.
\end{equation*}
\end{proof}

\subsection{Proof  of Proposition~\ref{pro:local-lin}}\label{prf:pro:local-lin}
\begin{proof}$~$
\begin{itemize}

\item Claim~\ref{pro:local-lin1} easily follows  from Premise~\ref{prem_C1}. Indeed, we have that $\ak\to a$, $\bk\to b$ and $\gak \to \gamma$. Since $a,b\in[0,1]^2$ and $\gamma\in]0,2/L[$ thus  $(1+\ak)\to(1+a)$ , $\gak(1+\bk)\to\gamma(1+b)$ and $\gak\bk\to\gamma b$. This implies that
\[
 (1+\ak)\Id-\gak(1+\bk)\nabla_x^2f(\xsol,\theta)\to  (1+a)\Id-\gamma(1+b)\nabla_x^2f(\xsol,\theta),
\]
and 
\[
-\ak\Id+\gak\bk\nabla_x^2f(\xsol,\theta)\to -a\Id+\gamma b\nabla_x^2f(\xsol,\theta).
\]
hence $M_k\to M.$

\item For Claim~\ref{pro:local-lin2}, let $\sigma$ be an eigenvalue of $M$ related to the eigenvector $\begin{pmatrix}r_1\\r_2\end{pmatrix}$ then we recall that $M=\begin{pmatrix} (a-b)\Id+(1+b)G_{\theta} &-(a-b)\Id-bG_{\theta} \\ \Id  & 
		0\end{pmatrix}$.  Let $\eta_{\theta}$ be an eigenvalue of $G_{\theta},$  we have 			
\begin{align*}
M\begin{pmatrix}r_1\\r_2\end{pmatrix}=\begin{pmatrix} (a-b)r_1+(1+b)G_{\theta}r_1 -(a-b)r_2-bG_{\theta}r_2 \\ 
		r_1\end{pmatrix}=\sigma\begin{pmatrix}r_1\\r_2\end{pmatrix}.
\end{align*} 
The second line of this identity means that $r_1=\sigma r_2$, we plug this in the first identity to get the following quadratic equation in $\sigma$
\begin{equation}\label{eq:quad-sigma}
\sigma^2-\Ppa{\pa{a-b}+\pa{1+b}\eta_{\theta}}\sigma+\pa{a-b}+b\eta_{\theta}=0.
\end{equation}		
Now, we have to solve \eqref{eq:quad-sigma} in terms of $\sigma$, this make look tedious at first sight but it is similar to the proof of \cite[Proposition~17]{liang_activity_2017}. They  found that the eigenvalues  $\sigma$ satisfying  \eqref{eq:quad-sigma} are such that $\rho(M)=\abs{\sigma}<1$ if and only if Premise~\ref{prem_C2} holds true.

\item  Thanks to  Claim~\ref{pro:local-lin2}, to prove claim~\ref{pro:local-lin3}, we only have  to prove that their exist $K$ large enough such that
\[
\forall k\geq K, \quad X_{k+1}(\theta)-X^{\star}(\theta)=M\Ppa{X_{k}(\theta)-X^{\star}(\theta)}.
\]
Due to Remark~\ref{rem:conv-iters}, the sequence $\seq{X_{k}(\theta)}$ converges to the solution $X^{\star}(\theta)$  there exists  $K\in\bbN$ sufficiently large such that $X_{k}(\theta)$ is close enough to $X^{\star}(\theta)$. For $k\geq K,$ we have that  $\forall\theta\in\Theta$, 
\begin{align*}
\xkp\pa{\theta}-\xsol\pa{\theta}&=\yak-\xsol-\gak\nabla f\pa{\ybk,\theta},
\end{align*}
By the Premise~\ref{prem_A}, we have  that the function $f$ is a $C^2-$smooth function, thus it is also twice differentiable in the extended sense according to \cite[Theorem~13.2]{rockafellar_variational_1998}. Hence, by \cite[Definition~13.1]{rockafellar_variational_1998} (Definition of twice diffrentiability in the extended sense), we have
\begin{align*}
\xkp\pa{\theta}-\xsol\pa{\theta}&=\yak-\xsol-\gak\BPa{\nabla_x^2 f\pa{\xsol,\theta}\pa{\ybk-\xsol}+o\pa{\normm{\ybk-\xsol}}},
\end{align*}
where we have used the optimality condition that $\nabla f\pa{\xsol,\theta}=0$. Let us observe that 
\begin{align*}
\normm{\ybk-\xsol}&=\normm{\pa{1+\ak}\Ppa{\xk-\xsol}-\ak\Ppa{\xkm-\xsol}},\\
&\leq 2\Ppa{\normm{\xk-\xsol}+\normm{\xkm-\xsol}},\\
&\leq 4\eps,
\end{align*} 
where we have chosen $K$ large enough such that $\forall k\geq K $ we get that $\xk$ and $\xkm$ are $\eps-$ sufficiently close to $\xsol$. We get  similarly that $\yak$ is sufficiently close to $\xsol,$ hence we have $o\pa{\normm{\ybk-\xsol}}=0.$ 

Now, we expend $\yak$ and $\ybk$ to get that  
\begin{align*}
\xkp\pa{\theta}-\xsol\pa{\theta}=&\Big[\pa{1+\ak}\Id-\gak\pa{1+\bk}\nabla^2_xf\pa{\xsol,\theta}\Big]\pa{\xk-\xsol}\\&-\Big[\ak\Id-\gak\bk\nabla^2_xf\pa{\xsol,\theta}\Big]\pa{\xkm-\xsol}\\
=&\Big[\pa{\ak-\bk}\Id-\gak\pa{1+\bk}G^k_{\theta}\quad -\pa{\ak-\bk}\Id+\gak\bk G^k_{\theta}\Big]\Ppa{X_k(\theta)-X^{\star}(\theta)},
\end{align*}
where we have denoted $\forall k\in\bbR^n,  G^k_{\theta}=\Id-\gak\nabla^2_xf\pa{\xsol,\theta}$. We get by adding an appropriate line  the following
\begin{align*}
X_{k+1}(\theta)-X^{\star}(\theta)=&\begin{pmatrix}\pa{\ak-\bk}\Id-\gak\pa{1+\bk}G^k_{\theta}&-\pa{\ak-\bk}\Id+\gak\bk G^k_{\theta}\\
\Id&0
\end{pmatrix}\Ppa{X_k(\theta)-X^{\star}(\theta)},\\
=&\BPa{M+\begin{bmatrix}M_{k,1}\\0\end{bmatrix}+\begin{bmatrix}M_{k,2}\\0\end{bmatrix}}\Ppa{X_k(\theta)-X^{\star}(\theta)},
\end{align*}
where we have denoted 
\begin{align*}
&M_{k,1}\eqdef[(1+b)\pa{G^k_{\theta}-G_{\theta}},-b\pa{G^k_{\theta}-G_{\theta}}],\\
&M_{k,2}\eqdef[\Ppa{\pa{\ak-\bk}-\pa{a-b}}\Id+\pa{\bk-b}G^k_{\theta},-\Ppa{\pa{\ak-\bk}-\pa{a-b}}\Id-\pa{\bk-b}G^k_{\theta}].
\end{align*}
Finally, we have 
\begin{align*}
\forall k\in\bbN,\normm{X_{k+1}(\theta)-X^{\star}(\theta)}\leq &\rho(M)\normm{X_{k}(\theta)-X^{\star}(\theta)}+\normm{M_{k,1}\Ppa{X_{k}(\theta)-X^{\star}(\theta)}}\\
&+\normm{M_{k,2}\Ppa{X_{k}(\theta)-X^{\star}(\theta)}},\\
\leq& \rho(M)\normm{X_{k}(\theta)-X^{\star}(\theta)},
\end{align*}
where we used the Lemma~\ref{lem:inter-conv} which state that 
\[
\normm{M_{k,2}\Ppa{X_{k}(\theta)-X^{\star}(\theta)}}=\normm{M_{k,1}\Ppa{X_{k}(\theta)-X^{\star}(\theta)}}=o\Ppa{\normm{X_{k}(\theta)-X^{\star}(\theta)}}=0.
\]
\end{itemize}
\end{proof}

\section{Toolbox for  differentiation}\label{sec:toolbox-ad}
Consider an infinite interative process, according to Definition~\ref{def:codelist} given by the  triplet  $\calJ=\Ppa{\seq{\Phi_k},\Theta,\xo}$. We suppose that the generated sequence $\seq{\xk(\cdot)}$ is differentiable and moreover  derivative stable according to Definition~\ref{def:der-sta}. By the differentiability hypothesis  and the rule of derivation of composed functions, we have 
\begin{equation}\label{eq:pdcf-rule}
\forall k \in \bbN, \quad\partial_{\theta}\xkp(\theta)=J_x\Phi_k(\xk(\theta),\theta)\partial_{\theta}\xk(\theta)+J_{\theta}\Phi_k(\xk(\theta),\theta).
\end{equation}
 As suggested by \eqref{eq:pdcf-rule}, it turns out that that the convergence of the derivative depends on the Jacobians operator $J_{x}\Phi_k$ and $J_{\theta}\Phi_k$. Before we state the main Theorem,  let make the following premise. 
 \begin{premise}\label{prem_apend} $~$
\begin{enumerate}[label=(D.\arabic*)]
 \item \label{prem_D1} The sequence of iterative maps $\seq{\Phi_k}$ converges to a certain function $\Phi$,
 \item\label{prem_D2}  there exists a limit function $\avx(\theta)$ such that:  $\forall \theta\in\Theta,\quad\avx(\theta)=\Phi(\avx(\theta), \theta),$
 \item \label{prem_D3} $J_x\Phi_k\to J_x\Phi\qandq J_{\theta}\Phi_k\to J_{\theta}\Phi,$
 \item \label{prem_D4} $\forall\theta\in\Theta, \quad \rho\Ppa{J_x\Phi\pa{\avx(\theta),\theta}}<1.$
\end{enumerate}
\end{premise} 
\begin{remark}\label{rem:prem_apend}$~$
The first premise means that the sequence of the iterative maps $\seq{\Phi_k}$
 has a limit which we denote $\Phi$, then the  second premise  is also natural since we are examining if the derivative has a limit. The third premise is more constructive and is deduced from \eqref{eq:pdcf-rule}. Since, we want the sequence of derivative to have a limit, it sufficient in this setting to make the   hypothesis that  partial Jacobians sequences  also converges. The last condition is an invertibility condition on the the squared matrice $J_{x}\Phi(\avx(\theta),\theta)$ to get  an explicit formula for $\partial_{\theta}\avx(\theta).$
\end{remark}
 
 We can state the convergence theorem as  the following result.
 
\begin{theorem}\label{thm:deriv-stable} Let consider the following infinite iterative process $\calJ=\Ppa{\seq{\Phi_k},\Theta,\xo}$ according to Definition~\ref{def:codelist},  where $\Theta$ is an open subset of  $\bbR^m$ and $\avx(\cdot)$ the differentiable limit of $\seq{\xk}.$ 

Under the Premise~\ref{prem_apend}, we have that the sequence $\seq{\xk}$ is derivative-stable and moreover we can write the limit derivative as: 
\begin{centerbox}{black}{}y
\begin{equation}\label{eq:deriv-limit}
\forall \theta\in\Theta, \quad \partial_{\theta}\avx\pa{\theta}=\Ppa{\Id-J_x\Phi\pa{\avx(\theta),\theta}}^{-1}J_{\theta}\Phi\pa{\avx(\theta),\theta}.
\end{equation}
\end{centerbox}
\end{theorem}

\section{Proofs for the differentiation of inertial methods }\label{prf:ad-ima}

\subsection{Proof of Theorem~\ref{thm:gim-stable}}\label{prf:thm:gim-stable}
\begin{proof}
Let first notice that from Lemma~\ref{lem:def-inf},  $\calA$ is an infinite iterative process. Therefore, the proof of Theorem~\ref{thm:gim-stable} will consist in applying Theorem~\ref{thm:deriv-stable}. To achieve this goal, we will prove that our infinite iterative process $\calA$  satisfies Premise~\ref{prem_apend}. 
\paragraph{Part 1} \label{thm:gim-stable-part1}$~$  We recall that
\[
	F_k(X,\theta)= F_k(x,y,\theta)=\begin{pmatrix}
\yak(x,y)-\gak\nabla f(\ybk(x,y),\theta)\\
x
\end{pmatrix}.
\]
Using Lemma~\ref{lem:LipFk}, for all $k\in\bbN, F_k$ is a continuously differentiable function. By  Premise~\ref{prem_C1}  $\ak\to a$, $\bk\to b$ and $\gak \to \gamma$ thus $\yak\pa{x,y}\to y_{a}\pa{x,y}\eqdef x+a\pa{x-y}$  and $\ybk\pa{x,y}\to y_{b}\pa{x,y}\eqdef x+b\pa{x-y}$. This implies that $\yak(x,y)-\gak\nabla f(\ybk(x,y),\theta)\to y_a(x,y)-\gamma\nabla f(y_b(x,y),\theta)$. Consequently, we get that $\forall \pa{X,\theta}=(x,y,\theta)\in\bbR^{n}\times\bbR^{n}\times\Theta$,
\[
F_k(X,\theta)\to F(X,\theta)= \begin{pmatrix}
y_a(x,y)-\gamma\nabla f(y_b(x,y),\theta)\\
x
\end{pmatrix},
\]
which prove Premise~\ref{prem_D1}.
\paragraph{Part 2}\label{thm:gim-stable-part2}$~$
Thanks to the global convergence of the iterates generated by our scheme  Proposition~\ref{pro:conv-iters} and Remark~\ref{rem:conv-iters}, the sequence $\seq{X_k}$ converges to $X^{\star}(\theta)=\begin{pmatrix}\xsol(\theta)\\\xsol(\theta)\end{pmatrix}.$ Hence, we have 
\begin{align*}
 F(X^{\star}(\theta),\theta)= \begin{pmatrix}
y_a(\xsol(\theta),\xsol(\theta))-\gamma\nabla f(y_b(\xsol(\theta),\xsol(\theta)),\theta)\\
\xsol(\theta)
\end{pmatrix}= \begin{pmatrix}
\xsol(\theta)-\gamma\nabla f(\xsol(\theta),\theta)\\
\xsol(\theta)
\end{pmatrix}=X^{\star}(\theta),
\end{align*}
where we have used the optimality condition $\nabla f(\xsol(\theta),\theta)=0,$ which prove Premise~\ref{prem_D2}.
\paragraph{Part 3}\label{thm:gim-stable-part3}$~$ For any $(X,\theta)=(x,y,\theta)\in\bbR^{n}\times\bbR^{n}\times\Theta$ we have that 
\[
	 	J_1F_k(X,\theta)= \begin{pmatrix} (1+\ak)\Id-\gak(1+\bk)\nabla_x^2f(\ybk(x,y),\theta) &-\ak\Id+\gak\bk\nabla_x^2f(\ybk(x,y),\theta) \\ \Id  & 
		0\end{pmatrix}
\]
and 
\[
		 	J_2F_k(X\pa{\theta},\theta)= \begin{pmatrix} -\gak\nabla_{x\theta}^2f(\ybk(x,y),\theta)  \\  
		0\end{pmatrix}.
\]
Similarly to Part 1, we use Premise~\ref{prem_C1} and we have that $\ak\to a$, $\bk\to b$ and $\gak \to \gamma$. Since $a,b\in[0,1]^2$ and $\gamma\in]0,2/L[$ thus  $(1+\ak)\to(1+a)$ , $\gak(1+\bk)\to\gamma(1+b)$ and $\gak\bk\to\gamma b$. This implies that
\[
 (1+\ak)\Id-\gak(1+\bk)\nabla_x^2f(\ybk(x,y),\theta)\to  (1+a)\Id-\gamma(1+b)\nabla_x^2f(y_b(x,y),\theta),
\]

\[
-\ak\Id+\gak\bk\nabla_x^2f(\ybk(x,y),\theta)\to -a\Id+\gamma b\nabla_x^2f(y_b(x,y),\theta),
\]
and 
\[
-\gak\nabla_{x\theta}^2f(\ybk(x,y),\theta)\to-\gamma\nabla_{x\theta}^2f(y_b(x,y),\theta).
\]
Consequently, we get  
\[
J_1F_k(X,\theta)\to J_1F(X,\theta)= \begin{pmatrix} (1+a)\Id-\gamma(1+b)\nabla_x^2f(y_b(x,y),\theta) &-a\Id+\gamma b\nabla_x^2f(y_b(x,y),\theta)\\ \Id  & 
		0\end{pmatrix}
\]
and 
\[
J_2F_k(X\pa{\theta},\theta)\to J_2F(X\pa{\theta},\theta)  \begin{pmatrix} -\gamma\nabla_{x\theta}^2f(y_b(x,y),\theta)  \\  
		0\end{pmatrix},
\]
 which prove the Premise~\ref{prem_D3}. 
 
 \paragraph{Part 4}\label{thm:gim-stable-part4}$~$
For the last part, we applied Proposition~\ref{pro:local-lin}-\ref{pro:local-lin2}, since premise~\ref{prem_A}, \ref{prem_B}, \ref{prem_C} and the strong convexity hypothesis with respect to $x$ hold, we get that  the spectral radius of $M=J_1F(X^{\star},\theta)$ is such that $\rho(M)<1.$
 
 \paragraph{Conclusion} $~$
 From Part 1, 2, 3 and 4 we applied  Theorem~\ref{thm:deriv-stable} to get that the sequence $\seq{X_k}$ is derivative stable. Moreover we have  by \eqref{eq:deriv-limit}  the derivative limit is written as $\forall\theta\in\Theta,$
 \begin{align*}
 \partial_{\theta}X^{\star}\pa{\theta}&=\Ppa{\Id-J_1F(X^{\star}\pa{\theta},\theta)}^{-1}J_2F(X^{\star}\pa{\theta},\theta),\\ 
 &=\begin{pmatrix} -a\Id+\gamma(1+b)\nabla_x^2f(\xsol(\theta),\theta) &a\Id-\gamma b\nabla_x^2f(\xsol\pa{\theta},\theta) \\ -\Id  & 
		\Id\end{pmatrix}^{-1}\hspace{-0.2cm}\begin{pmatrix} -\gamma\nabla_{x\theta}^2f(\xsol,\theta)  \\  
		0\end{pmatrix},
 \end{align*}
 which concludes the proof of the Theorem.
\end{proof}

\subsection{Proof  of Theorem~\ref{thm:conv-rat-der}}\label{prf:thm:conv-rat-der}

\begin{proof}
Let us first recall that by the rule of derivation of composed functions, we have the formula \eqref{eq:pdcf-rule} 
\[
 \forall k\in\bbN, \quad \partial_{\theta}X_{k+1}(\theta)=J_1F_k(X_k(\theta),\theta)\partial_{\theta}X_k(\theta)+J_2F_k(X_k(\theta),\theta).
\]
Thus on the limit, we have that 
\[
\partial_{\theta}X^{\star}(\theta)=J_1F(X^{\star}(\theta),\theta)\partial_{\theta}X^{\star}(\theta)+J_2F(X^{\star}(\theta),\theta).
\]
This yields to the following $\forall k\in\bbN,$
\begin{align*}
\partial_{\theta}X_{k+1}(\theta)-\partial_{\theta}X^{\star}(\theta)=&J_1F_k(X_k(\theta),\theta)\partial_{\theta}X_k(\theta)+J_2F_k(X_k(\theta),\theta)-J_1F(X^{\star}(\theta),\theta)\partial_{\theta}X^{\star}(\theta)\\
&-J_2F(X^{\star}(\theta),\theta),\\
=&\BPa{J_1F_k(X_k(\theta),\theta)+J_1F(X^{\star}(\theta),\theta)}\Ppa{\partial_{\theta}X_k(\theta)-\partial_{\theta}X^{\star}(\theta)}\\
&+\BPa{J_2F_k(X_k(\theta),\theta)-J_2F(X^{\star}(\theta),\theta)}+\calE_k,
\end{align*}
where we set 
\[
\calE_k\eqdef J_1F_k(X_k(\theta),\theta)\partial_{\theta}X^{\star}\pa{\theta}-J_1F(X^{\star}(\theta),\theta)\partial_{\theta}X_k\pa{\theta}.
\]
Henceforth, we have the following bound 
\begin{align*}
\normm{\partial_{\theta}X_{k+1}(\theta)-\partial_{\theta}X^{\star}(\theta)}\leq&\normm{J_1F_k(X_k(\theta),\theta)+J_1F(X^{\star}(\theta),\theta)}\normm{\partial_{\theta}X_k(\theta)-\partial_{\theta}X^{\star}(\theta)}\\
&+\normm{J_2F_k(X_k(\theta),\theta)-J_2F(X^{\star}(\theta),\theta)}+\normm{\calE_k},\\
\leq&\BPa{\normm{J_1F_k(X_k(\theta),\theta)}+\normm{J_1F(X^{\star}(\theta),\theta)}}\normm{\partial_{\theta}X_k(\theta)-\partial_{\theta}X^{\star}(\theta)}\\
&+\normm{J_2F_k(X_k(\theta),\theta)-J_2F(X^{\star}(\theta),\theta)}+\normm{\calE_k},
\end{align*}
It remains to properly bound each term of the previous inequality. 

Let us recall that $\forall k\in\bbN, J_1F_k(X_k(\theta),\theta)$ is a continuous operator since $\forall k\in \bbN, F_k$ is $C^1-$smooth. Moreover, we have that $X_k(\theta)\to X^{\star}(\theta)$ and we have prove in Part~\ref{thm:gim-stable-part3} the sequence $J_1F_k(\cdot,\cdot)\to J_1F(\cdot,\cdot)$ which implies that $J_1F_k(X_k(\theta),\theta)\to J_1F(X^{\star}(\theta),\theta)$ hence their exists $K_1$ large enough such that we have
\[
\normm{J_1F_k(X_k(\theta),\theta)}\leq\normm{J_1F(X^{\star}(\theta),\theta)}+\eps_1,
\]
An analogous reasoning yields to the fact that their exists $K_2$ large enough such that
\[
\normm{J_2F_k(X_k(\theta),\theta)-J_2F(X^{\star}(\theta),\theta)}\leq \eps_2,
\]

 Let us consider now $\normm{\calE_k}$, we have 
\begin{align*}
\normm{\calE_k}&=\normm{ J_1F_k(X_k(\theta),\theta)\partial_{\theta}X^{\star}\pa{\theta}-J_1F(X^{\star}(\theta),\theta)\partial_{\theta}X_k\pa{\theta}},
\end{align*}
 we get  that $\normm{\calE_k}\to0.$ This means that their exits $K_3$ large enough such that 
 \[
 \normm{\calE_k}\leq\eps_3, \forall k\geq K_3. 
 \]
By summing everything up, taking $K=\max\Ba{K_1,K_2,K_3}$ and $\eps=\min\Ba{\eps_1,\eps_2,\eps_3}$ we obtain that for all $k\geq K$ we have
\begin{align*}
\normm{\partial_{\theta}X_{k+1}(\theta)-\partial_{\theta}X^{\star}(\theta)}\leq&2\normm{J_1F(X^{\star}(\theta),\theta)}\normm{\partial_{\theta}X_k(\theta)-\partial_{\theta}X^{\star}(\theta)}+\eps\Ppa{2+\normm{\partial_{\theta}X^{\star}(\theta)}},\\
\leq&2\rho\pa{M}\normm{\partial_{\theta}X_k(\theta)-\partial_{\theta}X^{\star}(\theta)}+\eps\Ppa{2+\normm{\partial_{\theta}X_k(\theta)-X^{\star}(\theta)}},
\end{align*}
this concludes the  proof of this Theorem. 

\end{proof}

	\end{appendices}

\begin{acknowledgments}
The  author would like to thank Samuel Vaiter for his helpful comments and French National Research  ANR for funding the PEPR PDE-AI Grant.
\end{acknowledgments}

\bibliographystyle{plain}
{
\bibliography{biblio}
}

\end{document}